\documentclass[12pt]{amsart}%
\usepackage[body={17cm,21cm}]{geometry}

\usepackage[utf8x]{inputenc}
\usepackage{cite}
\usepackage{graphicx}
\usepackage{mathrsfs}
\usepackage{amscd,latexsym,amsthm,amsfonts,amssymb,amsmath,amsxtra}
\usepackage{comment}
\usepackage[colorlinks, urlcolor=blue,  citecolor=blue]{hyperref}
\usepackage[all]{xy}%
\providecommand{\U}[1]{\protect\rule{.1in}{.1in}}
\RequirePackage{amsmath}
\RequirePackage{amssymb}

\newcommand*{\dprime}{^{\prime\prime}\mkern-1.0mu}
\newcommand{\WA}{\textbf{(WA)}}
\newcommand{\SA}{\textbf{(SA)} }
\newcommand{\APSA}{\textbf{(APSA)} }

\newcommand{\BA}{{\mathbb {A}}}

\newcommand{\BF}{{\mathbb {F}}}
\newcommand{\BG}{{\mathbb {G}}}

\newcommand{\BN}{{\mathbb {N}}}

\newcommand{\BQ}{{\mathbb {Q}}}
\newcommand{\BR}{{\mathbb {R}}}

\newcommand{\BZ}{{\mathbb {Z}}}
\newcommand{\CA}{{\mathcal {A}}}

\newcommand{\CC}{{\mathcal {C}}}
\renewcommand{\CD}{{\mathcal {D}}}
\newcommand{\CE}{{\mathcal {E}}}

\newcommand{\CG}{{\mathcal {G}}}

\newcommand{\CO}{{\mathcal {O}}}
\newcommand{\CP}{{\mathcal {P}}}

\newcommand{\CR}{{\mathcal {R}}}

\newcommand{\CT}{{\mathcal {T}}}
\newcommand{\CU}{{\mathcal {U}}}
\newcommand{\CV}{{\mathcal {V}}}

\newcommand{\CY}{{\mathcal {Y}}}

\newcommand{\RA}{{\mathbf {A}}}

\newcommand{\codim}{{\mathrm{codim}}}

\newcommand{\GL}{{\mathrm{GL}}}

\newcommand{\ord}{{\mathrm{ord}}}

\renewcommand{\mod}{\ \mathrm{mod}\ }

\newcommand{\Res}{{\mathrm{Res}}}

\newcommand{\SL}{{\mathrm{SL}}}
\newcommand{\Spec}{{\mathrm{Spec}}}

\newcommand{\vol}{{\mathrm{vol}}}

\newcommand{\sbt}{\subset}

\newcommand{\bk}{\bar{k}}

\newcommand{\prs}{\operatorname{pr}^S}

\numberwithin{equation}{section}

\theoremstyle{remark}

\newtheorem{defi}{\rm{\textbf{Definition}}}[section]

\newtheorem{rem}[defi]{\rm{\textbf{Remark}}}
\newtheorem{rems}[defi]{\rm{\textbf{Remarks}}}

\newtheorem*{rem*}{\rm{\textbf{Remark}}}
\newtheorem*{rems*}{\rm{\textbf{Remarks}}}
\newtheorem{ques}[defi]{\rm{\textbf{Question}}}

\theoremstyle{plain}

\newtheorem{thm}[defi]{\rm{\textbf{Theorem}}}

\newtheorem{cor}[defi]{\rm{\textbf{\textbf{Corollary}}}}

\newtheorem{lem}[defi]{\rm{\textbf{Lemma}}}

\begin{document}

\title[Arithmetic purity]
 {Arithmetic purity of strong approximation for semisimple simply connected groups}

\author{Yang Cao}
\author{Zhizhong Huang}

\address{Yang CAO
\newline Leibniz Universit\"at Hannover
\newline Welfengarten 1, 30167 Hannover, Germany}
\email{yangcao1988@gmail.com; yang.cao@math.uni-hannover.de}

\address{Zhizhong Huang
\newline Leibniz Universit\"at Hannover
\newline Welfengarten 1, 30167 Hannover, Germany}
\email{zhizhong.huang@yahoo.com}

\date{\today.}



\begin{abstract}
In this article we establish the arithmetic purity of strong approximation for certain semisimple simply connected linear algebraic groups and their homogeneous spaces over a number field $k$. For instance, for any such group $G$ and for any open subset $U$ of $G$ with $\codim(G\setminus U, G)\geqslant 2$, we prove that 
\begin{enumerate}
	\item  if $G$ is $k$-simple and $k$-isotropic, then $U$ satisfies strong approximation off any finite number of places;
	\item  if $G$ is the spin group of a non-degenerate quadratic form which is not compact over archimedean places, then $U$ satisfies strong approximation off all archimedean places.
\end{enumerate} 
 As a consequence, we prove that the same property holds for affine quadratic hypersurfaces. 
 Our approach combines a fibration method with subgroup actions developed for induction on the codimension of $G\setminus U$, and an affine linear sieve which allows to produce integral points with almost-prime polynomial values.

\end{abstract}

\maketitle
	\setcounter{tocdepth}{1} 
\tableofcontents


\section{Introduction}
\subsection{Background}
Algebraic varieties over global fields whose set of rational points is dense in the adelic space with respect to the adelic topology are said to verify \emph{strong approximation} (cf. \cite[Def. 2.9]{Wit16}). 
 As from the definition of adelic spaces (cf. e.g. \cite{Conrad}), the study of strong approximation naturally involves integral points. They are generally believed to be more intricate to tackle than rational points, regarding our present knowledge, and it seems that more evidence is needed to formulate satisfactory conjectures.
  For affine varieties, the adelic topology is in general finer than the product topology of local points. The density of rational points in the latter is referred to \emph{weak approximation}, and these two topologies agree if the ambient variety is proper.
If the topology is ``too'' fine, then it can happen that strong approximation fails even for the simplest examples like $\BA^1$, since number fields are diagonally embedded  in their adeles as closed discrete lattices. One needs to modify the topology via ``forgetting'' a finite number of places, say $S$, i.e., to consider the topology induced by adelic points without $S$ components, in which we refer to the property \emph{strong approximation off} $S$ (cf. Definition~\ref{def:sa} (1)).

Algebraic groups were the first class of varieties for which strong approximation was systematically investigated. It is Eichler \cite{Eic38} who pioneered this study for some classical groups. Since the 60s, several authors established that semisimple simply connected algebraic groups verify strong approximation. The number field case is mainly due to Kneser \cite{Kne65} and Platonov \cite{Plat2} (cf. \cite[Thm.~7.12]{PR}). \begin{thm}[Kneser-Platonov]\label{rmk:KneserPlatonov}
	A non-trivial semisimple simply connected group $G$ over a number field $k$ satisfies strong approximation off $S$, a finite subset of places of $k$, if and only if $G^\prime(k_S):=\prod_{v\in S} G^\prime(k_v)$ is not compact for every $k$-simple factor $G^\prime$ of $G$.
\end{thm}
The above non-compactness condition on $G^\prime$ is equivalent to the existence of a place $v\in S$ such that $G^\prime(k_v)$ is not compact. We refer to the book \cite[Chap.~7]{PR} for further details and references.

It is a natural problem to study how these density properties vary under birational modifications. First of all, by an ultrametric version of the implicit function theorem (cf. \cite[Lem.~3.5.5]{Serre} and \cite[\S3.1]{PR}), the weak approximation property is a birational invariant amongst smooth varieties. 
However, for strong approximation off any finite number of places, this is not true any more even for $\BA^1$, as it fails already for $\BG_{\operatorname{m}}$. Moreover, being geometrically simply connected is a necessary condition in order for varieties over number fields to verify strong approximation, as was first pointed out by Min\v{c}hev \cite[Thm. 1]{Minchev} (see also \cite[\S8.4.6]{Poonen}). 
On the other hand, consider the variety $\BA^1\times\BG_{\operatorname{m}}$ over $\BQ$. This variety does not satisfy strong approximation off $\BR$, but this failure is explained by the Brauer-Manin obstruction, thanks to the Poitou-Tate duality (cf. \cite[\S1.2]{Harari}). However, Xu and the first author proved that the defect of strong approximation off $\BR$ for the variety $\BA^1\times\BG_{\operatorname{m}}\setminus \{(0,1)\}$ cannot be explained by any étale Brauer-Manin obstruction (cf. \cite[Ex.~5.2]{CX}). See \cite[\S2.7]{Wit16} for more examples. All these phenomena suggest that the behaviour of integral points can be sensitive to removing closed subsets, while Brauer groups are insensitive to removing subschemes of codimension two (the purity of Brauer groups, cf. \cite[\S6]{Grothendieck}).

In a reverse direction, if a smooth variety verifies strong approximation (when there is no Brauer-Manin obstruction), one expects that any other smooth variety differing by a closed subvariety of codimension at least 2 inherits this property.
The following question has been raised first by Wittenberg at the AIM workshop ``Rational and integral points on higher-dimensional varieties'' in 2014 (cf. \cite[Problem 6]{AIM14}) and also by Colliot-Thélène and Xu, respectively, on different occasions. See also \cite[\S2.7 Question 2.11]{Wit16} and \cite[Question 1.2]{CLX}.
\begin{ques}\label{q:purity}
	Let $X$ be a smooth variety over a number field which satisfies strong approximation off a finite set of places $S$.  Does any Zariski open subset $U$ of $X$ also satisfy this property, whenever $\codim(X\setminus U,X)\geqslant 2$?
\end{ques}
Following \cite[p.~336]{CLX}, we shall say that such $X$ verifies the \emph{arithmetic purity of strong approximation off $S$} (Def. \ref{def:sa} (2)).
One piece of evidence towards a positive answer to Question \ref{q:purity} is supported by the purity of branch locus due to Zariski--Nagata--Grothendieck (cf. \cite[Thm.~5.2.13, Cor.~5.2.14]{Sza}), which states that the algebraic fundamental group is insensitive to removing codimension two subschemes. So a positive answer to Question \ref{q:purity} is compatible with Min\v{c}hev's observation.

Let us summarize the state of the art in the literature. Wei \cite[Lem.~1.1]{Wei14} and, independently, Xu and the first author \cite[Prop. 3.6]{CX} proved that Question \ref{q:purity} has a positive answer for the affine space $\BA^n$.  This result was applied to show  \cite[Thm.~1.2]{CX} that strong approximation with Brauer\textendash Manin obstruction holds for toric varieties, 
and afterwards, it was generalised by Liang, Xu and the first author \cite[Thm.~1.1]{CLX}, who answered Question \ref{q:purity} in the affirmative for semisimple simply connected linear algebraic groups that are \emph{quasi-split}. They also made an important reduction \cite[Thm.~1.3]{CLX} -- a positive answer to Question \ref{q:purity} for general semisimple simply connected linear algebraic groups implies the arithmetic purity of strong approximation with Brauer-Manin obstruction for their homogeneous spaces with connected stabilizers and no non-constant invertible regular functions.
  
  Motivated by this progress, the objective of this paper is to give further positive evidence towards Question \ref{q:purity} as a continuation of the above investigations.
More precisely, we give a positive answer to Question \ref{q:purity} for semisimple simply connected almost simple linear algebraic groups  which are \emph{isotropic}, and for 
spin groups of any non-degenerate quadratic forms with more than three variables, when $S$ is the set of all archimedean places.

\subsection{Notation and terminology}
In this article, the letter $k$ denotes a number field unless otherwise specified. We fix an algebraic closure $\bar{k}$ of $k$. We denote by $\Omega_{k}$ the set of non-trivial places of $k$ and denote by $\infty_{k}$ the set of archimedean places. 
For each place $v\in \Omega_{k}$,  $k_{v}$ denotes the completion of $k$.
We denote by $\CO_k$ (resp. $\CO_{v}$) the ring of integers of $k$ (resp. $k_v$ for $v\in\Omega_{k}\setminus \infty_{k}$). 
The ring of adeles of $k$ is denoted by $\RA_{k}$.
For any finite subset $S$ of $\Omega_k$, the ring of $S\cup \infty_{k}$-integers of $k$ is denoted by $\CO_{k,S}$. 
The ring of adeles without $S$-components is denoted by $\RA_{k}^{S}$.

 A \emph{variety} over $k$ is a separated scheme of finite type over $k$.
 For $X$ a $k$-variety, the set of $k$-rational points $X(k)$ embeds diagonally into $X(\RA_{k})$, and we denote by $\prs:X(\RA_{k})\to X(\RA_{k}^{S})$ the natural projection between sets of adelic points of $X$. We define $X(k_S):=\prod_{v\in S}X(k_v)$. 
When $S=\infty_{k}$ we write $k_S=k_\infty$ for simplicity.
For each $x\in X$ (resp. $v\in\Omega_{k}$), $k(x)$ (resp. $k(v)$) denotes its residue field.
For any field extension $K/k$, we write $X_K:=X\times_k K$. 
\begin{defi} \label{def:sa} 
Let $X$ be a variety over a number field $k$ and $S$ be a finite subset of $\Omega_k$.
\begin{enumerate}
	\item We say that $X$ satisfies \emph{strong approximation off $S$}, abbreviated as \SA off $S$, if the diagonal embedding of $X(k)$ is dense in $\prs (X(\RA_{k}))$;
	\item 
	Let $c\in\BN_{\geqslant 1}$.
	We say that $X$ satisfies \emph{the arithmetic purity of strong approximation off $S$ of codimension $c$} if any open subvariety of $X$ whose complement is of codimension $\geqslant c$ satisfies strong approximation off $S$. In particular, when $c=2$, we simply say that $X$ satisfies \emph{the arithmetic purity of strong approximation off $S$}, abbreviated as \APSA off S.
\end{enumerate}
\end{defi}
Similarly we shall denote ``weak approximation'' by \textbf{(WA)}.

Our terminology follows standard textbooks on linear algebraic groups (e.g. \cite{Bo}, \cite{CGP} and \cite{Hum}).
Let $G$ be a connected linear algebraic group over $k$.  
We say that $G$ is $k$-\emph{isotropic} if $G$ has $k$-rank greater than one, that is, if $G$ contains a non-trivial $k$-split torus. Otherwise, we say that $G$ is $k$-\emph{anisotropic} if it does not contain any $k$-split torus (cf. \cite[\S20.1]{Bo}, \cite[p.~635]{CGP} and \cite[\S34.4]{Hum}). We say that $G$ is $k$-\emph{quasi-split} if it contains a Borel subgroup  over $k$ (cf. \cite[\S35.1]{Hum}). We call $G$ \emph{(almost) $k$-simple} if it does not contain any non-commutative non-trivial proper connected normal subgroups defined over $k$ (cf. \cite[p.~168]{Hum}).

Let $q$ be a non-degenerate quadratic form over $k$ in $n$ variables. The \emph{spin group} $\operatorname{Spin}(q)$ is the universal (double) covering of $\operatorname{SO}(q)$, the special orthogonal group associated to $q$. If $n\geqslant 3$, the group $\operatorname{Spin}(q)$ is connected, semisimple and simply connected. See \cite[\S5.1]{CTX09}.

\subsection{Main results} 

We now state several main theorems of this paper regarding Question \ref{q:purity}.
Our first result resolves all $k$-simple and isotropic cases off an arbitrary place.
\begin{thm}\label{thm:isotropic}
	Let $k$ be a number field and $v_0\in \Omega_k$.
	Let $G$ be a semisimple simply connected $k$-simple linear algebraic group.
	If $G$ is $k$-isotropic, then $G$ satisfies \APSA off $v_0$.
\end{thm}
\begin{rems}\label{rmk:BorelTits}
		\hfill
		\begin{enumerate}

		\item For any semisimple linear algebraic $k$-group $G$, according to the Borel-Tits theory (cf. \cite[Rem.~C.3.10]{CGP}), we know that $G$ is $k$-isotropic if and only if $G$ contains a closed $k$-subgroup isomorphic to $\BG_{\operatorname{a}}$, and equivalently, $G$ contains a parabolic subgroup over $k$ different from $G$. Thus being $k$-quasi-split implies being $k$-isotropic.
		
		\item In general, the notion of being isotropic is much weaker than being quasi-split. For example, there are many $k$-isotropic spin groups of quadratic forms in $\geqslant 4$ variables which are not $k$-quasi-split (cf. \cite[\S 23.4]{Bo}).\footnote{A concrete example is the family of spin groups of non-degenerate quadratic forms in $11$ variables, in which there is only one quasi-split $\BQ$-form $\operatorname{Spin}(6,5)$. We thank one anonymous referee for kindly pointing out this to us.} Therefore Theorem \ref{thm:isotropic} generalises \cite[Thm.~1.1]{CLX} and covers many new cases.
		
		\item When $k$ is a totally imaginary field, Theorem \ref{thm:isotropic} covers all semisimple simple connected $k$-simple groups which are not of $A_n$-type. Indeed $k$-groups of $A_n$-type are the only $k$-anisotropic cases (cf. \cite[Thm.~6.25, pp.~352]{PR}).
		\end{enumerate}
\end{rems}
At the AIM workshop, Wittenberg \cite[Question 6]{AIM14} proposed affine quadratic hypersurfaces as first interesting cases worthy of study. This is answered affirmatively by our next result.

\begin{thm}\label{cor:anisotropic}
	Let $k$ be a number field.
	Let $q(x_1,\cdots,x_n)$ be a non-degenerate quadratic form over $k$ with $n\geqslant 3$ variables and $G:=\operatorname{Spin}(q)$ be the corresponding spin group.
	If $q$ is isotropic over an archimedean place, 
	then 
	
	(i) $G$ satisfies \APSA off $\infty_{k}$;
	
	(ii) for any $a_0\in k^{\times}$, the affine quadric defined by $q(x_1,\cdots,x_n)=a_0$ in $\BA^{n}$ 
	satisfies \APSA off $\infty_{k}$ if $n\geqslant 4$.
\end{thm}

The following result is the technical core for the proof of Theorem \ref{cor:anisotropic}. It applies in particular to spin groups.

\begin{thm}\label{thm:anisotropic}
	Let $k$ be a number field.
	Let $G$ be a semisimple simply connected $k$-simple linear algebraic group.
	If $G$ contains a closed subgroup $G'\sbt G$ with $G'$ a three-dimensional semisimple simply connected linear algebraic group over $k$ such that $G'(k_{\infty})$ is not compact, 
	then $G$ satisfies \APSA off $\infty_{k}$.
\end{thm}

\begin{rems}\hfill
	\begin{enumerate}
		\item When $n=3$, the affine quadric in Theorem \ref{cor:anisotropic} (ii) can have non-trivial Brauer group and hence Brauer-Manin obstruction to integral points, while if $n\geqslant 4$ there is no such obstruction (cf. \cite[\S1.2]{Bor} and \cite[\S5.6,~\S5.8]{CTX09}). We actually show that \APSA off $\infty_{k}$ \emph{with Brauer-Manin obstruction} (cf. \cite[Def.~2.1]{CLX}) holds for the affine quadratic surface $q(x_1,x_2,x_3)=a_0$ in $\BA^3$.
		\item When $k$ is not totally real, the non-compactness conditions in Theorems \ref{cor:anisotropic} and \ref{thm:anisotropic} become vacuous.
	\end{enumerate}
\end{rems}




\subsection{Strategies and organization of the paper}
Our strategy relies on a combination of a fibration method with group actions and an affine linear sieve due to Sarnak and his collaborators. The former is developed in Section \ref{se:isotropic} and applied to the proof of a more general version of Theorem \ref{thm:isotropic}, i.e. Theorem \ref{thm:reduction}, and to the anisotropic cases in Section \ref{se:anisotropic}. Roughly speaking, having fixed some closed subset $D$ of a semisimple group $G$, a non-trivial subgroup $H$ action (up to conjugation) can ``increase'' the codimension of $D$ (Lemma \ref{isolem2}). Then the proof of Theorem \ref{thm:reduction} is achieved by considering the fibration $G\to G/H$ via an induction argument on the codimension of $D$.
A large part of Section \ref{se:anisotropic} is devoted to proving Theorem \ref{th:mani1} concerning three-dimensional anisotropic groups. After a careful choice of fibration into families of tori, a major difficulty that remains to overcome is to find integral points that can ``avoid'' the given codimension two subset modulo infinitely many primes. In doing so, we make use of torus actions, and we appeal to the aforementioned sieve argument which allows to produce almost-prime values with only large prime divisors, the latter constituting Section \ref{se:sieve}. More explanations on the proof of Theorem \ref{th:mani1} are postponed to Section \ref{se:overview41}. 
Finally, in Section \ref{se:affquad}, combining Theorem \ref{th:mani1} with the fibration argument established in Section \ref{se:isotropic}, we prove Theorems \ref{cor:anisotropic} and \ref{thm:anisotropic}.

We hope that the ideas developed in this text would provide new insights into Wittenberg's question \ref{q:purity} and have further application.
\subsection{Some more notation and conventions}

Let $G$ be an algebraic group over $k$. Unless otherwise specified, subgroups of $G$ and irreducible components of subsets of $G$ are defined over $k$. We shall write e.g. $H^\prime\subset G_K$ to denote a subgroup $H^\prime$ of $G_K$ defined over a field extension $K$ of $k$. 

By convention, $\dim(\varnothing)=-1$.

 The letter $p$ is always reserved for a prime number, and $\varepsilon$ denotes an arbitrarily small positive parameter that can be rescaled by constant multiples. For two real-valued functions $f$ and $g$ with $g$ non-negative, Landau's symbol $f=O(g)$ and Vinogradov's symbol $f\ll g$ (or $g\gg f$) are used interchangeably, meaning that there exists $C>0$ such that $|f|\leqslant Cg$.  Any dependence of the implied constant $C$ on extra parameters will be specified as subscripts. We write $f\asymp g$ if $f\ll g$ and $g\ll f$ hold simultaneously.
 We write $\mu(\cdot)$ for the Möbius function.
\section{Isotropic case}\label{se:isotropic} 
Our ultimate goal is to prove Theorem \ref{thm:isotropic} via establishing the following general result. 

\begin{thm}\label{thm:reduction}
	Let $k$ be a number field and $S\sbt \Omega_k$ be a finite set of places.
	Let $G$ be a semisimple simply connected $k$-simple linear algebraic group over $k$.
	If $G$ contains a non-trivial connected closed subgroup $H$ that satisfies \APSA off $S$,  then $G$ also satisfies \APSA off $S$.
\end{thm}

\subsection{Subgroup action on closed subsets}
We start by proving two technical lemmas. We are grateful to Harpaz for useful exchanges from which our arguments here are inspired.
In this section, let $k$ be a number field, $G$ be a non-trivial semisimple simply connected $k$-simple linear algebraic group over $k$.
Let us fix in this part $H\sbt G$ a non-trivial connected closed subgroup. For $g\in G$, let $H_g:=gH_{k(g)}g^{-1}$. 

\begin{lem}\label{isolem1}
Let $E,D$ be two proper non-empty closed subsets of $G$ defined over $k$ with $E\sbt D$ and $\dim(E)=\dim(D)$.
Then the set 
$$\Psi_{H,E,D}:=\{g\in G:\ H_g\cdot E_{k(g)}\not\subset D_{k(g)}\}$$
is a non-empty open subset of $G$ over $k$.
\end{lem}

\begin{proof}
Firstly we prove that $\Psi_{H,E,D}\sbt G$ is open.
Let $p_1: G\times H\times E\to G$ be the first projection (hence flat and open) and define
\begin{align*}
	\phi:  G\times H\times E&\to G\\
	(g,h,e)&\mapsto ghg^{-1}e.
\end{align*}
Then $\Psi_{H,E,D}=p_1(\phi^{-1} (G\setminus D))\subset G$ is open, since $D$ is closed.

Next we prove that $\Psi_{H,E,D}\sbt G$ is non-empty. 

In the first place we treat the case where $E=D$. Namely, we show that for any non-empty proper closed subset $E\subset G$, we have $\Psi_{H,E,E}\neq\varnothing$. 
Let $H'\sbt G$ be the normal closure of $H$, 
i.e. $H'_{\bk}\sbt G_{\bk}$ is the closed subgroup generated by $gHg^{-1}$ for all $g\in G(\bk)$ and $H'_{\bk}$ can be defined over $k$ because it contains all Galois conjugates of its $\bk$-components.
 Then $H'\sbt G$ is the minimal closed normal subgroup containing $H$.
 We claim that $H'$ is connected.
 Indeed, the action of the connected group $G$ on $H'$ by conjugation preserves the identity component $H^{\dprime}$ of $H^\prime$ 
  and hence $H^{\dprime}$ is also a closed normal subgroup of $G$ containing $H$, and therefore $H^{\dprime}=H'$ since $H'$ is minimal. This proves the claim.
 Since $G$ is $k$-simple and $H^{\prime}$ is normal and connected with $\dim H^\prime\geqslant \dim H>0$ by the assumption that $H$ is non-trivial, we deduce that $H^{\prime}=G$.  
If $\Psi_{H,E,E}=\varnothing$, then $G_{\bk}=G_{\bk}\cdot E_{\bk}=H^\prime_{\bk}\cdot E_{\bk}\sbt E_{\bk}$, which contradicts $E\neq G$.

Now we return to the general case of two closed sets $E,D$ as in the assumption of the lemma. Let $E'\sbt E$ be an irreducible component of the same dimension as $E$. 
From the previous paragraph, we have shown $\Psi_{H,E',E'}\neq\varnothing$. We now show
$$\Psi_{H,E',E'}\subset \Psi_{H,E,D},$$
so that $\Psi_{H,E,D}\neq\varnothing$.
Indeed, for any $g\in \Psi_{H,E',E'}$, we have $ H_g\cdot E'_{k(g)}\not\subset E'_{k(g)}$. So there exists an irreducible $k(g)$-component $E^{\dprime}$ of $E'_{k(g)}$ of the same dimension as $E'_{k(g)}$, such that $H_g\cdot E^{\dprime}\not\subset E^{\dprime}$. Since $\overline{H_g\cdot E^{\dprime}}$ is irreducible, and $E^{\dprime}\subset H_g\cdot E^{\dprime}$, one has $\dim(H_g\cdot E^{\dprime})>\dim(E^{\dprime})=\dim(E^{\prime})=\dim (E)=\dim(D)$.
Thus $H_g\cdot E'_{k(g)}\not\subset D_{k(g)}$ and \emph{a fortiori} $H_g\cdot E_{k(g)}\not\subset D_{k(g)}$. This shows that $g\in\Psi_{H,E,D}$, as desired.
\end{proof}

\begin{lem}\label{isolem2}
	 Let $D$ be a proper non-empty closed subset of $G$ over $k$.
Then there exist $g\in G(k)$ and a closed subset $D'\sbt D$ over $k$ such that: $$\dim(D')< \dim(D),\quad H_g\cdot D'=D',\quad
 \text{and} \quad H_g\cdot (G\setminus D)=G\setminus D'.$$
\end{lem}

\begin{proof}
Let $\{D_i\}_{i\in I}$ be the irreducible components of $D$ of the same dimension, and for each $i\in I$, let $\Psi_{H,D_i,D}$ be the set defined in Lemma \ref{isolem1}.
Since $G(k)$ is Zariski dense in $G$,\footnote{Such groups also satisfy \textbf{(WA)} by \cite[Prop.~7.9]{PR}.} and each $\Psi_{H,D_i,D}$ is open and dense in $G$, we have $$G(k)\cap (\cap_{i\in I} \Psi_{H,D_i,D})\neq\varnothing.$$
Having fixed $g\in G(k)\cap (\cap_{i\in I} \Psi_{H,D_i,D})$, let us define $$D':=G\setminus (H_g\cdot (G\setminus D)).$$
Then it is clear that $G\setminus D'= H_g\cdot (G\setminus D)$. Since $H_g\cdot(G\setminus D)\subset G\setminus D$, we get $D'\sbt D$. If there exists $d\in D^\prime$ such that $H_g\cdot d\not\subset D^\prime_{k(d)}$, then there exists $h\in H_g$ such that $h\cdot d \in G\setminus D^\prime=H_g\cdot(G\setminus D)$. This implies $d\not \in D$, a contradiction. Consequently $H_g\cdot D'=D'$. Since $H_g\cdot (G\setminus D)$ is open and defined over $k$, $D^\prime$ is closed and defined over $k$. 
By the definition of $\Psi_{H,D_i,D}$, we have $H_g\cdot D_i\not\subset D$ for all $i\in I$. So one has $D_i\not\subset D'$ for any $i\in I$, for otherwise $H_g\cdot D_i\subset H_g\cdot D^\prime =D^\prime\subset D$, a contradiction.
Thus $\dim(D')<\dim(D)$.
\end{proof}

\subsection{Proof of Theorem \ref{thm:reduction} -- a fibration method with subgroup actions}
By the assumption that $H$ satisfies the \APSA off $S$, one has that $H(k_S)$ is not compact.
Then $G(k_S)$ is also not compact.
Let $$\pi: G\to Y:=G/H$$ be the quotient map.
Let $D\sbt G$ be a closed subset over $k$ of codimension $c\geqslant 2$ and $U:=G\setminus D$. 
Now we prove that $U$ satisfies \SA off $S$.
By induction, we may assume $G$ satisfies \APSA off $S$ of codimension $c+1$, the initial case being $c=\dim G$ and following from Theorem \ref{rmk:KneserPlatonov}.

By Lemma \ref{isolem2}, upon replacing $H\sbt G$ by some $H_g$ for $g\in G(k)$ if necessary, we may assume there exists a closed subset $D'\sbt D$ over $k$ such that
$H\cdot D'=D'$, $\dim(D')<\dim(D)$ and $\pi(U)=Y\setminus \pi(D')$.
Let us define
$$U_1:=G\setminus D',\ \ V_1:=Y\setminus \pi(D'),\ \  \pi_{U_1}:=\pi|_{U_1}: U_1\to V_1\ \  \text{and}\ \  \pi_U:=\pi|_{U}: U\to V_1.$$
Then $\pi_{U_1}$ is a $H$-torsor and $\pi_U$ is smooth surjective with geometrically integral fibres. We get the following commutative diagram:
$$\xymatrix{U\ar@{->>}[rd]_{\pi_U}\ar@{^{(}->}[r]&U_1\ar@{^{(}->}[r]\ar@{->>}[d]^{\pi_{U_1}}&G\ar@{->>}[d]^{\pi}&D\ar@{_{(}->}[l]&D'\ar@{_{(}->}[l]\ar[d]\\&V_1\ar@{^{(}->}[r]&Y&&\pi(D').\ar@{_{(}->}[ll]
}$$
By \cite[Prop. 3.5]{CLX}, there exists an open subset $V_0$ of $V_1$ over $k$ such that for any $y\in V_0$, we have 
\begin{equation}\label{eq:codim}
\codim(D_y,G_y)\geqslant\codim(D,G)\geqslant c,
\end{equation} where $D_y,G_y$ are the fibres of $\pi$ over $y$.
Define $U_0:=\pi_{U}^{-1}(V_0)$.

For any non-empty open subset $W\sbt U(\RA_k^{S})$, we want to show that $U(k)\cap W\neq\varnothing$.
By \cite[Thm. 4.5]{Conrad}, the set $\pi_U(W)\sbt V_1(\RA_k^{S})$ is open, and by continuity,
the set $$W_1:=\pi_{U_1}^{-1}(\pi_U(W))\sbt U_1(\RA_k^{S})$$ is open.
Since $\codim(D',G)\geqslant \codim(D,G)+1\geqslant c+1$ by induction hypothesis, the variety $U_1$ satisfies \SA off $S$, so $U_1(k)\cap W_1\neq\varnothing$.
By \cite[Prop.~2.3 (2)]{CLX}, we have $U_0(k)\cap W_1\neq \varnothing$.
Let us choose
$ x_0\in U_0(k)\cap W_1$ and let $y_0:=\pi(x_0)$. Then $x_0\in U_{y_0}(k)\subset G_{y_0}(k)$, and thus $G_{y_0}\simeq H$. Since  $y_0\in V_0(k)\cap \pi_{U_1}(W_1)=V_0(k)\cap \pi_U(W)\neq \varnothing$,
the set $W\cap U_{y_0}(\RA_k^{S})$ is a non-empty open adelic neighbourhood of $U_{y_0}(\RA_k^{S})$, where $U_{y_0}$ is the fibre of $\pi_U$ over $y_0$.
 Since $H$ satisfies \APSA off $S$, the open variety $U_{y_0}=G_{y_0}\setminus D_{y_0}$ satisfies \SA off $S$, so $U_{y_0}(k)\cap (W\cap U_{y_0}(\RA_k^{S}))\neq\varnothing$.
Therefore we have proven $U(k)\cap W\neq\varnothing$. This finishes the proof.
\qed

\begin{rem}
	The reason why our induction process from the $c=2$ case to the $c=1$ case fails is that the condition $\codim(D_{y_0},G_{y_0})\geqslant 1$ deduced from \eqref{eq:codim} is insufficient to produce rational points in $W\cap U_{y_0}(\RA_k^{S})$ by the definition of \APSA (Definition \ref{def:sa} (2)).
\end{rem}

\subsection{Proof of Theorem \ref{thm:isotropic}}

	By Borel-Tits' theory (Remark \ref{rmk:BorelTits} (i)), the assumption that the group $G$ being $k$-isotropic is equivalent to saying that $G$ contains a closed subgroup isomorphic to $\BG_{\operatorname{a}}$ over $k$.
	The statement now follows from Theorem \ref{thm:reduction} applied to $\BG_{\operatorname{a}}\sbt G$, for which \SA off any place is known by the Chinese Remainder Theorem. \qed

\section{Almost prime polynomial values}\label{se:sieve}
In this section we study quantitative growth of almost-prime integral polynomial values on semisimple simply connected groups. 
We shall prove Theorem \ref{th:primepoint}, asserting the finiteness of the \emph{saturation number} \cite[p.~361]{NS} of a non-zero regular function 
evaluated on a given principal congruence subgroup.
For further application, we derive Corollary \ref{cor:primepoint} as a key tool in the proof of Theorem \ref{th:mani1}.
Our approach is based on the \emph{affine linear sieve} introduced in \cite{NS}.

\subsection{Main result}\label{se:mainsieve}
In this section we fix $G$ a semisimple simply connected linear algebraic group over $\BQ$. Let us fix an embedding $G\hookrightarrow\GL_{n,\BQ}$. We choose the integral model $\CG:=\overline{G} \subset\GL_{n,\BZ}$, i.e. the integral closure of $G$ in $\GL_{n,\BZ}$. The coordinate ring $\BQ[G]$ is a unique factorization domain. We fix a height function $\|\cdot\|$ on $\GL_{n}(\BZ)$. For example, we can define for $g=(g_{i,j})\in \GL_n(\BZ)$, $\|g\|:=\max_{1\leqslant i,j\leqslant n} |g_{i,j}|$. Let $$\Gamma:=\CG(\BZ),$$ and for any $\alpha\in\BN_{\geqslant 1}$, define 
$$\Gamma_{\alpha}:=\{g\in \CG(\BZ):\ g\equiv \operatorname{id}_\CG \mod \alpha\},$$
the principal congruence subgroup of level $\alpha$.

We now state the main result of this section. 

\begin{thm}\label{th:primepoint}
	Let $f\in\BQ[G]$ be a non-zero regular function and $\alpha_0\in\BN_{>1}$. Assume
	that $f$ takes integer values on $\Gamma_{\alpha_0}$.
	Let $N:=\gcd(f(\Gamma_{\alpha_0}))$ and $S \subset \Omega_{\BQ}\setminus\infty_{\BQ}$ be a finite set of places such that $f\in\BZ_S[\CG]$. Assume moreover that every $\BQ$-simple factor of $G$ is not compact over $\BR$. Then there exist $\beta_1>0,\beta_2>0$, depending only on $\CG$, $f$, $\alpha_0$, $S$ and $N$, such that for any large enough $T$,
	\begin{equation}\label{eq:loweralmostprime}
	\#\left\{g\in \Gamma_{\alpha_0}: \|g\|<T ,\gcd\left(f(g),\prod_{\substack{p\nmid\alpha_0 N,p\not\in S\\p<T^{\beta_1}}}p\right)=1 \right\}\gg  \frac{\# \{g\in \CG(\BZ): \|g\|<T\}}{(\log T)^{\beta_2}}.
	\end{equation}
\end{thm}
In this section, as well as in $\S$\ref{se:proofprimepoint} below, unless otherwise mentioned, all implied constants are allowed to depend on $\CG$, $\|\cdot\|$, $f$, $\alpha_0$, $S$ and $N$.
\begin{rems}\label{rmk:sieveremark}
	\hfill
	\begin{enumerate}
		\item Thanks to the works \cite[Thm.~1.2]{DRS}, \cite[Thm.~1]{Maucourant}, \cite[Thm.~2.7]{Gorodnik-Weiss}, we know that there exist rational numbers $a>0,b\geqslant 0$ such that 
		\begin{equation}\label{eq:asympab}
		\# \{g\in \CG(\BZ): \|g\|<T\}\asymp T^a(\log T)^b.
		\end{equation}
		See e.g. \cite[\S3.1]{NS} for an interpretation of the constants $a,b$. So Theorem \ref{th:primepoint} implies that the number of $\Gamma_{\alpha_0}$-lattice points $g$ of bounded height $T$ such that $f(g)$ is free of prime factors less than $T^{\beta_1}$, 
		except for those dividing $\alpha_0,S$ or $N$, goes to infinity as $T$ grows.
	
		\item  In \cite[Thm.~1.7, Cor.~1.8,~(4.32)]{NS}, under the assumption that $\gcd(f(\Gamma))=1$,\footnote{This is called \emph{weakly primitive} in \cite[p. 361]{NS}.} and that $f$ factors into $t(f)$ distinct irreducible factors over $\BQ$ which are also assumed to be absolutely irreducible, it is shown that for $\alpha_0=1$, \eqref{eq:loweralmostprime} has the expected magnitude of growth. Namely the power on $\log T$ is $t(f)$. Our result deals with arbitrary non-zero $f$ and also implies the finite saturation of the pair $(f,\Gamma_{\alpha_0})$, using e.g. \cite[Lem.~4.2]{NS}.
	\end{enumerate}
\end{rems}

\subsection{An application}

In this section, let $k$ be a number field and $G$ be a semisimple simply connected $k$-simple linear algebraic group over $k$.  We fix an embedding $G\sbt \GL_{n,k}$ and we choose the integral model $\CG$ to be the integral closure of $G$ in $\GL_{n,\CO_k}$.


Theorem \ref{th:primepoint} is sufficient to deduce the following technical result. The way in which we formulate it is tailored to our application in Section \ref{se:anisotropic}.

\begin{cor}\label{cor:primepoint}
	Let $S'\sbt (\Omega_{\BQ}\setminus \infty_{\BQ})$, $S\sbt (\Omega_k\setminus \infty_k)$ be two finite subset of places such that $\CO_{k,S}$ is a principal ideal domain and finite \'etale over $\BZ_{S'}$.
	Let $$\Phi:=\prod_{v\in S}\Phi_v\times \prod_{v\notin S\cup \infty}\CG(\CO_v)$$ be a subgroup of $G(\RA_k^{\infty_k})$, where $\Phi_v\sbt G(k_v)$ is an open compact subgroup for each $v\in S$.
	Let $f\in \CO_{k,S}[\CG]$ be a regular function such that $f(\CG(\CO_v))\cap \CO_v^{\times}\neq\varnothing$ for any $v\in \Omega_{k}\setminus (S\cup \infty_k)$. Assume that  $G(k_{\infty})$ is not compact.
	Then there exists an integer $r_0\geqslant 1$ depending only on $\CG,\Phi,f,S,S^\prime$ and verifying the following property: for any $M>0$, there exists $g_0\in \CG(\CO_{k,S})\cap \Phi$ 
	such that $f(g_0)$ has at most $r_0$ prime factors in $\CO_{k,S}$, and the cardinalities of their residue fields are larger than $M$.
\end{cor}
The assumption that $\CO_{k,S}$ is finite étale over $\BZ_{S'}$ means precisely that the extension $k/\BQ$ is unramified outside of $S^\prime$, and $S$ is the set of places over $S^\prime$.
\begin{proof}
	Let us consider 
	$$G^\prime:= \Res_{k/\BQ}G\sbt\GL_{n[k:\BQ],\BQ},\quad\CG^\prime:=\Res_{\CO_k/\BZ}\CG,\quad f^\prime:=N_{k/\BQ}(f).$$ 
	The group $G^\prime$ is semisimple simply connected and $\BQ$-simple by \cite[Prop.~A.5.14]{CGP}, and $\CG^\prime$ is a smooth integral model of $G^\prime$ over $\BZ_{S^\prime}$ by \cite[\S7.6~Prop.~5~\&~6]{BLR}.
	 We identify the group $\Phi$ as a subgroup of $G^\prime(\RA_\BQ^{\infty_{\BQ}})$. Take an integer $\alpha_1$ of the form $\prod_{v\in S'}p_v^{m_v}$ with $m_v\in\BN_{\geqslant 1}$ large enough so that the diagonal image of the principal congruence subgroup $\Gamma_{\alpha_1}$ of $\CG^\prime(\BZ)$ in $G^\prime(\RA_\BQ^{\infty_{\BQ}})$ is contained in  $\Phi$. Since $G(k_\infty)$ is not compact, so is $G^\prime(\BR)$. Since $f\in \CO_{k,S}[\CG]$, we have $f^\prime\in\BZ_{S'}[\CG^\prime]$, and by assumption, $f^\prime$ is not identically zero on $\CG^\prime(\BZ_{S^\prime})$.
	 
	 The determinant character $G^\prime\hookrightarrow\GL_{n[k:\BQ]}\xrightarrow{\det}\BG_{\operatorname{m}}$ is trivial since $G^\prime$ is semisimple, so $G^\prime$ embeds into $\SL_{n[k:\BQ],\BQ}\subset\BA^{(n[k:\BQ])^2}$, where we regard $\GL_{n[k:\BQ],\BQ}\subset\BA^{(n[k:\BQ])^2}$.
	 So $$\BQ[G^\prime]=\BQ\left[X_1,\cdots,X_{(n[k:\BQ])^2}\right]/I,$$ the ideal $I$ being generated by finitely polynomials in $\BZ[X_1,\cdots,X_{(n[k:\BQ])^2}]$ including $\det(\mathbf{X})-1$ where $\det(\mathbf{X})$ is the determinant polynomial. Therefore we can choose $$F\in \BZ_{S^\prime}\left[X_1,\cdots,X_{(n[k:\BQ])^2}\right]$$ a lift of $f^\prime$ in such a way that it has minimal degree.
	 And we can multiply $F$ by an integer $R$ whose prime divisors are in $S^\prime$, so that $RF$ has integer coefficients, and in particular $Rf^\prime$ takes integer values on $\Gamma_{\alpha_1}$.
	 Moreover, the hypothesis $f(\CG(\CO_v))\cap \CO_v^{\times}\neq\varnothing$ for any $v\notin S\cup \infty_{k}$ implies that $N:=\gcd(Rf^\prime(\Gamma_{\alpha_1}))$ factorises into prime numbers in $S^\prime$. 
	
	We now apply Theorem \ref{th:primepoint} to the group $G^\prime$ and to the function $Rf^\prime$ evaluated on $\Gamma_{\alpha_1}$. For any $T$ large enough such that $T^{\beta_1}\geqslant M$, we can find $g_0\in\Gamma_{\alpha_1}$ such that $\|g_0\|\leqslant T$ and, according to \eqref{eq:loweralmostprime}, any prime divisor of $Rf^\prime(g_0)$ is either $>T^{\beta_1}$, or $\in S^\prime$. 
	On the other hand, there exists $\vartheta>0$ depending only on $R$ and the degree of $F$ such that for any $g\in\GL_n(\BZ)$ with $\|g\|< T$, we have $|RF(g)|\leqslant T^\vartheta$. Hence for any $g\in\CG^\prime(\BZ)$ with $\|g\|< T$, we also have $|Rf^\prime(g)|\leqslant T^\vartheta$. So the number of distinct prime factors of $Rf^\prime(g_0)$ which are $\geqslant T^{\beta_1}$ is $\leqslant \lfloor\vartheta/\beta_1\rfloor +1$. Returning to the function $f$ and regarding $g_0\in \Phi$, we see that $g_0$ is in the diagonal image of $\CG(\CO_{k,S})$, and we have shown that any $v\in\Omega_{k}\setminus (S\cup\infty_{k})$ dividing $f(g_0)$ satisfies $\#k(v)\geqslant T^{\beta_1}\geqslant M$, and the total number of such $v$ is $\leqslant [k:\BQ](\lfloor \vartheta/\beta_1\rfloor +1)$. Therefore we may take $r_0=[k:\BQ](\lfloor \vartheta/\beta_1\rfloor +1)$ to conclude.
\end{proof}

\subsection{A combinatorial sieve}
Our central analytic input is an elementary combinatorial sieve of Brun-Selberg type, based on the so-called ``Fundamental Lemma''.
It gives estimates for the growth of an integer sequence free of moderately small prime divisors.
If the elements of this sequence have polynomial growth, a direct consequence is that each element sieved out has a uniformly bounded number of prime divisors, hence is``almost-prime''.

Let $\CP$ be a set of prime numbers.
For $z>0$, define $$\CP(z):=\prod_{p<z,p\in\CP} p.$$ 
Let $\CA=(a_i)_{i\in I}$ be a finite sequence of integers indexed by $I$.
	Define $$X:=\#\CA ~(=\# I).$$
	For $d\in\BN_{\geqslant 1}$, we define the subsequence $$\CA_d:=(a_i)_{i\in I_d}$$ of $\CA$ with $I_d:=\{i\in I:d\mid a_i\}\subset I$. 
	Our goal is to estimate the \emph{sifting function}
	$$S(\CA,\CP,z):=\#\{i\in I:\gcd(a_i,\CP(z))=1\}.$$
	We now formulate the following version of the ``Fundamental Lemma'', which is best suitable for our application. See for example \cite[Thm.~6.9, Cor.~6.10]{Friedlander-Iwaniec}, \cite[I.4~Thm.~3]{Tenenbaum}, and \cite[Thm. 7.2]{HR}.
		

\begin{thm}\label{thm:sieve}
	With the notation above, suppose that there exist a multiplicative arithmetic function $\omega:\BN_{\geqslant 1}\to\BR_{\geqslant 0}$ and constants $\kappa>0, A_1>1$ verifying the following hypotheses.
\begin{enumerate}
	\item 
	For any prime $p$, we have 
	$$0\leqslant\frac{\omega(p)}{p}<1;$$
	\item for any real numbers $w_1,w_2$ such that $2\leqslant w_1\leqslant w_2$, we have
	$$\prod_{\substack{w_1\leqslant p\leqslant w_2}}\left(1-\frac{\omega(p)}{p}\right)^{-1}\leqslant A_1\left(\frac{\log w_2}{\log w_1}\right)^\kappa.$$
\end{enumerate}
Then there exist constants $\lambda\in ~]0,(9\kappa)^{-1}[~,\tau>0$ depending only on $\kappa,A_1$ such that, for any finite sequence of integers $\CA$ and any real numbers $y\geqslant 2$ and $z\in[2,y^\lambda]$, we have
\begin{equation}\label{eq:sifting}
S(\CA,\CP,z)\geqslant \tau X\prod_{p\mid\CP(z)}\left(1-\frac{\omega(p)}{p}\right)+O\left(\sum_{\substack{d\leqslant y\\d\mid \CP(z)}}\left|\#\CA_d-\frac{\omega(d)}{d}X\right|\right),
\end{equation}
where the implied constant depends only on $\kappa,A_1$, and is independent of $\CA,y,z$.
\end{thm}
The quantity $\frac{\omega(d)}{d}X$ is considered as an approximation of the cardinality of $\CA_d$ when $d$ factorises into prime divisors in $\CP$, and
 $$\CR_\CA(d):=\#\CA_d-\frac{\omega(d)}{d}X$$
 is the remainder term. So Theorem \ref{thm:sieve} implies that $S(\CA,\CP,z)$ grows like $X\prod_{p\mid\CP(z)}\left(1-\frac{\omega(p)}{p}\right)$, provided that $\omega(\cdot)$ are $\CR_\CA(\cdot)$ are both small on average.
\subsection{The Lang-Weil estimate revisited}
Before proceeding to the proof of Theorem \ref{th:primepoint}, we recall a version of Lang-Weil estimate dealing with rational points on reductions of arbitrary varieties at finite places.

\begin{thm}[Lang-Weil]\label{co:Lang-Weil}
	Let $V$ be a variety over a number field $k$ and let $\CV$ be a model of $V$ over $\CO_{k,S}$ where $S\subset \Omega_k\setminus\infty_{k}$ is a finite set of places.  Then there exists a constant $C(\CV)>0$ depending only on $\CV$ such that for any prime ideal $\mathfrak{p}$ of $\CO_{K,S}$, 
	$$\#\CV(\BF_q)\leqslant C(\CV) q^{\dim V},$$
	where $\BF_q:=\CO_{k,S}/\mathfrak{p}$ is of cardinality $q$. If $\CV_{\BF_q}:=\CV\times_{\CO_{k,S}}\Spec(\BF_q)$ is geometrically irreducible, then $$\#\CV(\BF_q)=q^{\dim V}+O(q^{\dim V-\frac{1}{2}}),$$
	where the implied constant depends only on $\CV$.
\end{thm}
\begin{proof}
	We apply \cite[Thm.~7.7.1]{Poonen} with $X=\CV,Y=\Spec(\CO_{k,S})$.
\end{proof}


\subsection{Proof of Theorem \ref{th:primepoint}}\label{se:proofprimepoint}

We keep using the notation in $\S$\ref{se:mainsieve}. Most of our argument follows \cite[\S4]{NS}.

 Upon multiplying $f$ by some integer with prime factors in $S$ or dividing $\alpha_0$, we may assume that 
\begin{equation}\label{eq:asskey}
f\in\BZ[\CG], \quad \alpha_0\mid N, \quad \text{ and }\quad \forall p\in S,p\mid N,
\end{equation}
(recall that $N:=\gcd(f(\Gamma_{\alpha_0}))$) without affecting the validity of \eqref{eq:loweralmostprime}.
For $D\in\BN_{\geqslant 1}$, let $\Gamma_{\alpha_0}[D]$ (resp. $\Gamma[D]$) be the reduction modulo $D$ of $\Gamma_{\alpha_0}$ (resp. $\Gamma$) in $\GL_n(\BZ/D\BZ)$, and let $$\Gamma_{\alpha_0}^f[D]:=\{x\in\Gamma_{\alpha_0}[D]:f(x)\equiv 0\mod D\}.$$
Let $\CP:=\{p:p\nmid N\}$. For $d\in\BN_{\geqslant 1}$, let us define the arithmetic function 
\begin{equation}\label{eq:rhof}
\varrho_f(d):=\begin{cases}
\displaystyle\frac{d \#\Gamma_{\alpha_0}^f[dN]}{\#\Gamma_{\alpha_0}[dN]} &\text{ if } \gcd(d,N)=1; \\ 0 &\text{ otherwise}.
\end{cases}
\end{equation}
Let us define the finite sequence of integers
$$\CA:=(N^{-1}|f(g)|)_{g\in\Gamma_{\alpha_0}:\|g\|<T}$$
indexed by elements in $\Gamma_{\alpha_0}$ of bounded height.
Then we recall from \eqref{eq:asympab} that 
\begin{equation}\label{eq:X}
X:=\#\CA=\#\{g\in\Gamma_{\alpha_0}:\|g\|<T\}\asymp T^a(\log T)^b.
\end{equation} 

 For the rest of the proof, we first show that $\varrho_f(\cdot)$ is multiplicative in $\S$\ref{se:rhomult}, and verify in $\S$\ref{se:sievecond} that $\varrho_f(\cdot)$ satisfies the conditions (1) (2) in the statement of Theorem \ref{thm:sieve}. We then give uniform estimates for the error term $\CR_\CA(\cdot)$ in \S\ref{se:univerror}, and finally apply Theorem \ref{thm:sieve} to $\CA$ in $\S$\ref{se:applysieve}.

\subsubsection{Multiplicativity}\label{se:rhomult}
By assumption that $G(\BR)$ has no compact $\BQ$-factor, Theorem \ref{rmk:KneserPlatonov} implies that $G$ satisfies strong approximation off $\BR$, so $\Gamma$ is dense in $\prod_{v\in\Omega_{\BQ}\setminus\infty_{\BQ}}\CG(\BZ_v)$. 
In particular for any $d_1,d_2\in\BN_{\geqslant 1}$ with $\gcd(d_1,d_2)=1$, the reduction map $\Gamma\to\Gamma[d_1]\times\Gamma[d_2]$ is surjective. This shows that $\Gamma_{\alpha_0}[d]=\Gamma[d]$ if $\gcd(d,\alpha_0)=1$ (by taking $d_1=\alpha_0,d_2=d$, and note that $\Gamma_{\alpha_0}=\ker(\Gamma\to\Gamma[\alpha_0])$), and so
 $\Gamma_{\alpha_0}\subset\Gamma$ is dense in $\prod_{\substack{v\in\Omega_{\BQ}\setminus\infty_{\BQ},v\nmid \alpha_0}}\CG(\BZ_v)$.
 
It follows that, if $d=d_1d_2\in\BN_{\geqslant 1}$ with $\gcd(d,N)=\gcd(d_1,d_2)=1$, thanks to \eqref{eq:asskey}, the reduction map $\Gamma_{\alpha_0}\to \Gamma_{\alpha_0}[d]\times \Gamma_{\alpha_0}[N]$ is also surjective, and hence
\begin{equation}\label{eq:rhomult1}
\Gamma_{\alpha_0}[dN]\simeq \Gamma_{\alpha_0}[d_1]\times  \Gamma_{\alpha_0}[d_2]\times \Gamma_{\alpha_0}[N]\subset \GL_n(\BZ/d_1\BZ)\times \GL_n(\BZ/d_2\BZ)\times\GL_n(\BZ/N\BZ).
\end{equation}
As in the proof of Corollary \ref{cor:primepoint}, we fix a lift $f_1\in \BZ[\BA^{n^2}]$ of $f$.
Then for any $D\in\BN_{\geqslant 1}$ and any $x\in \Gamma_{\alpha_0}[D]$, we have $x\in\Gamma_{\alpha_0}^f[D]$ if and only if $f_1(x)\equiv 0\mod D$. So by the Chinese Remainder Theorem,
\begin{equation}\label{eq:rhomult2}
\Gamma_{\alpha_0}^f[d]\simeq\Gamma_{\alpha_0}^f[d_1]\times  \Gamma_{\alpha_0}^f[d_2]\times \Gamma_{\alpha_0}^f[N].
\end{equation} 
Moreover \begin{equation}\label{eq:rhomult3}
\Gamma_{\alpha_0}[N]=\Gamma_{\alpha_0}^f[N]
\end{equation} in this affine setting because $f$ has common divisor $N$ on $\Gamma_{\alpha_0}$.
The equalities \eqref{eq:rhomult1} \eqref{eq:rhomult2} \eqref{eq:rhomult3} above show that the arithmetic function $\varrho_f$ \eqref{eq:rhof} is multiplicative. (See also \cite[Prop.~4.1]{NS}.) 

\subsubsection{Verifying the sieve conditions}\label{se:sievecond}
Since the function $f/N$ takes integer values on $\Gamma_{\alpha_0}$ and satisfies $\gcd((f/N)(\Gamma_{\alpha_0}))=1$, we always have $\Gamma_{\alpha_0}[pN]\setminus\Gamma_{\alpha_0}^f[pN]\neq\varnothing$ for any prime $p\in\CP$. And we have $\varrho_f(p)=0$ whenever $p\not\in\CP$. Therefore $$0\leqslant\frac{\varrho_f(p)}{p}<1$$ for any prime $p$, so the sieve condition (1) in Theorem \ref{thm:sieve} holds.

Consider the closed subvariety $$V:=(f=0)\cap G.$$ We have $\dim V=\dim G-1$ by assumption that $f$ does not vanish identically on $G$.
Let $\CV$ be integral closure of $V$ in $\CG$. Let $S^\prime$ be a finite set of primes depending on $\CG$ and $N$ such that, $p\mid N\Rightarrow p\in S^\prime$ and $\CG_{\BF_p}:=\CG\times_{\BZ}\Spec(\BF_p)$ is smooth and geometrically irreducible for any $p\notin S^\prime$. In particular for any such $p$ we have $\Gamma_{\alpha_0}[p]=\Gamma[p]\simeq \CG(\BF_p)$ by strong approximation and Hensel's lemma.
Then thanks to the multiplicativity of $\varrho_f$ and applying the Lang-Weil estimate (Theorem \ref{co:Lang-Weil}) to $\CV$ and $\CG$, we obtain that, uniformly for any $p\notin S^\prime$,
\begin{equation}\label{eq:rhoCV}
\frac{\varrho_f(p)}{p}=\frac{\#\Gamma^f_{\alpha_0}[pN]}{\#\Gamma_{\alpha_0}[pN]}=\frac{\#\Gamma^f_{\alpha_0}[p]}{\#\Gamma_{\alpha_0}[p]}=\frac{\#\CV(\BF_p)}{\#\CG(\BF_p)}\leqslant \frac{\CC}{p},
\end{equation}
where the constant $\CC>0$ depends only on $\CV,\CG$.
Therefore for any real numbers $w_1,w_2$ verifying $\max(\CC,2,\max_{p\in S^\prime}p)< w_1\leqslant w_2$, there exists $A_0>0$ such that 
\begin{align*}
	\prod_{\substack{w_1\leqslant p\leqslant w_2}}\left(1-\frac{\varrho_f(p)}{p}\right)^{-1}\leqslant 	\prod_{\substack{w_1\leqslant p\leqslant w_2}}\left(1-\frac{\CC}{p}\right)^{-1}\sim A_0\left(\frac{\log w_2}{\log w_1}\right)^{\CC},
\end{align*}
by Mertens' formula (cf.\cite[(2.20)--(2.21)]{Friedlander-Iwaniec} or \cite[\S1.6~Thm.~7~\&~10]{Tenenbaum}).
Thus it suffices to choose $A_1=A_1(\CV,\CG,\CP)> \max(A_0,1)$ so that the sieve condition (2) in Theorem \ref{thm:sieve} is satisfied with $\kappa=\CC$ for any $2\leqslant w_1\leqslant w_2$, taking into account of the finitely many primes which are $\in S^\prime$ or $\leqslant \CC$.

\subsubsection{Uniform error term estimates}\label{se:univerror}
For any $d\in\BN_{\geqslant 1}$ with $\gcd(d,N)=1$, we now estimate the cardinality of subsequence $$\CA_d:=(N^{-1}|f(g)|)_{\substack{g\in\Gamma_{\alpha_0}:\|g\|<T,d\mid N^{-1}f(g)}},$$
and compare it with $\frac{\varrho_f(d)}{d}X$.

A key ingredient is \cite[Theorem 3.2]{NS}, which states that for any $\varepsilon>0$, uniformly for any $\xi\in\Gamma$ and any principal congruence subgroup $\Gamma_{\alpha}$,
\begin{equation}\label{eq:erroruniformestimate}
\#\{g\in\Gamma_{\alpha}:\|\xi\cdot g\|< T\}=\frac{\vol\{g\in\Gamma:\|g\|< T\}}{[\Gamma:\Gamma_{\alpha}]}+O_\varepsilon(T^{a-\frac{\theta}{(1+\dim G)}+\varepsilon}),
\end{equation}
where the constant $a$ is in \eqref{eq:asympab}, $\theta$ is given in \cite[Theorem 3.2]{NS}, and $\vol(\cdot)$ is associated to certain normalised Haar measure on $G(\BR)$. The error term above is uniform with respect to $\alpha$, and depends only on the height function $\|\cdot\|$.
This turns out to be crucial for the treatment of the error term in \eqref{eq:sifting}.

Thanks to \eqref{eq:asskey}, the reduction $\Gamma_{\alpha_0}\to \Gamma_{\alpha_0}[dN]$ has kernel $\Gamma_{dN}$. 
On applying \eqref{eq:erroruniformestimate} to both $\Gamma_{\alpha_0}$ and $\Gamma_{dN}$, we get that, uniformly for any $\xi\in\Gamma$,
\begin{align*}
	&\#\{g\in\Gamma_{dN}:\|\xi\cdot g\|<T\}\\=&\frac{\vol\{g\in\Gamma:\|g\|<T\}}{[\Gamma:\Gamma_{dN}]}+O_\varepsilon(T^{a-\frac{\theta}{(1+\dim G)}+\varepsilon})\\ =&\frac{[\Gamma:\Gamma_{\alpha_0}]}{[\Gamma:\Gamma_{dN}]}\left(\#\{g\in\Gamma_{\alpha_0}:\|g\|<T\}+O_\varepsilon(T^{a-\frac{\theta}{(1+\dim G)}+\varepsilon})\right)+O_\varepsilon(T^{a-\frac{\theta}{(1+\dim G)}+\varepsilon})\\ =&\frac{\#\{g\in\Gamma_{\alpha_0}:\|g\|<T\}}{[\Gamma_{\alpha_0}:\Gamma_{dN}]}+O_\varepsilon(T^{a-\frac{\theta}{(1+\dim G)}+\varepsilon}).
\end{align*}
Let us choose representatives $\xi_1,\cdots,\xi_{\#\Gamma_{\alpha_0}[dN]}\in\Gamma_{\alpha_0}$ of the group $\Gamma_{\alpha_0}/\Gamma_{dN}\simeq \Gamma_{\alpha_0}[dN]$,
and write $\overline{\xi_i}\in\Gamma_{\alpha_0}$ the image of $\xi_i$ modulo $dN$.
Since in this affine setting, the function $f$ takes constant residue modulo $dN$ on any coset $\xi_i\Gamma_{dN}$, 
we henceforth get

\begin{align*}
	\#\CA_d&=\#\{g\in\Gamma_{\alpha_0}:\|g\|<T,f(g)\equiv 0\mod dN\}\\
	&=\sum_{i=1}^{\#\Gamma_{\alpha_0}[dN]}\mathbf{1}_{dN
		\mid f(\overline{\xi_i})}\#\{g\in\Gamma_{dN}:\|\xi_i\cdot g\|<T\}\\
	&=\sum_{i=1}^{\#\Gamma_{\alpha_0}[dN]}\mathbf{1}_{dN
		\mid f(\overline{\xi_i})}\left(\frac{\#\{g\in\Gamma_{\alpha_0}:\|g\|<T\}}{[\Gamma_{\alpha_0}:\Gamma_{dN}]}+O_\varepsilon(T^{a-\frac{\theta}{(1+\dim G)}+\varepsilon})\right)\\
	&=\#\Gamma_{\alpha_0}^f[dN]\left(\frac{\#\{g\in\Gamma_{\alpha_0}:\|g\|<T\}}{[\Gamma_{\alpha_0}:\Gamma_{dN}]}+O_\varepsilon(T^{a-\frac{\theta}{(1+\dim G)}+\varepsilon})\right)\\
	&=\frac{\#\Gamma_{\alpha_0}^f[dN]}{\#\Gamma_{\alpha_0}[dN]}X+O_\varepsilon(\#\Gamma_{\alpha_0}^f[dN]T^{a-\frac{\theta}{(1+\dim G)}+\varepsilon}),
\end{align*}
where we recall $X$ \eqref{eq:X} and the implied constants in all error terms are independent of $d$. Therefore
\begin{equation}\label{eq:crca}
\begin{split}
\CR_\CA(d)&:=\#\CA_d-\frac{\varrho_f(d)}{d}X\\ &=O_\varepsilon(\#\Gamma_{\alpha_0}^f[dN]T^{a-\frac{\theta}{(1+\dim G)}+\varepsilon})=O_\varepsilon(\#\Gamma_{\alpha_0}^f[dN]X^{1-\frac{\theta}{a(1+\dim G)}+\varepsilon}).
\end{split}
\end{equation}

We now estimate $\#\Gamma_{\alpha_0}^f[dN]$ for $d\in\BN_{\geqslant 1}$ square-free with $\gcd(d,N)=1$. Thanks to \eqref{eq:rhomult2}, by employing the Lang-Weil estimate (Theorem \ref{co:Lang-Weil}), we obtain that, for any $\varepsilon>0$,

\begin{equation}\label{eq:upper}
\begin{split}
\#\Gamma_{\alpha_0}^f[dN]&\leqslant \#\GL_{n}(\BZ/N\BZ)\times \prod_{p\in S^\prime\cap \CP}\#\GL_n(\BF_p)\times \prod_{\substack{p\mid d,p\notin S^\prime}}\#\CV(\BF_p)\\ &\ll \prod_{\substack{p\mid d,p\notin S^\prime }} C(\CV)p^{\dim G-1} \ll_{\varepsilon} d^{\dim G-1+\varepsilon},
\end{split}
\end{equation}
 so that we can continue to bound \eqref{eq:crca} as
\begin{align*}
\CR_\CA(d)\ll_{\varepsilon} d^{\dim G-1+\varepsilon}X^{1-\frac{\theta}{a(1+\dim G)}+\varepsilon}.
\end{align*}

Therefore, with $y:=X^{\sigma}$, where $\sigma>0$ is chosen depending only $G$ such that $\gamma:=\frac{\theta}{a(1+\dim G)}-\sigma\dim G\in~]0,1[$, we get, uniformly for any $w\geqslant 2$,
$$\sum_{\substack{d\leqslant y\\d\mid \CP(w)}}|\CR_\CA(d)|\leqslant \sum_{\substack{d\leqslant y\\p\mid d\Rightarrow p\in\CP}}\mu^2(d)|\CR_\CA(d)|\ll_{\varepsilon} X^{1-\gamma+\varepsilon}.$$

\subsubsection{Applying the combinatorial sieve}\label{se:applysieve}
We deduce from Theorem \ref{thm:sieve} that, for the sequence $\CA$ with $T$ large enough so that $$z:=y^{\lambda}=X^{\lambda\sigma}\asymp T^{a\lambda\sigma}(\log T)^{b\lambda\sigma}>\max(\CC,2)$$ by \eqref{eq:X}, the sifting inequality \eqref{eq:sifting} now reads
\begin{align*}
S(\CA,\CP,z)&=\#\left\{g\in \Gamma_{\alpha_0}: \|g\|<T ,\gcd(f(g),\prod_{p<z,p\nmid  N}p)=1\right\}\\
&\gg X\prod_{p<z,p\nmid N}\left(1-\frac{\varrho_f(p)}{p}\right)\gg X\prod_{C(V)<p<z}\left(1-\frac{\CC}{p}\right)\\ &\asymp \frac{X}{(\log z)^{\CC}}\asymp \frac{\#\{g\in\CG(\BZ):\|g\|<T\}}{(\log T)^{\CC}},
\end{align*}
by \eqref{eq:rhoCV}, Mertens' formula and \eqref{eq:erroruniformestimate}. We conclude that \eqref{eq:loweralmostprime} holds for $\beta_1= a\lambda\sigma$ and $\beta_2=\CC$. This finishes the proof of Theorem \ref{th:primepoint}.
\qed

\begin{rem}
	A more sophisticated weighted sieve method (cf. e.g. \cite[Chap.~10]{HR}) should be able to enlarge $\beta_1$, and consequently give smaller admissible values for $r_0$ in Corollary \ref{cor:primepoint} (this is one of the ingredients for refinements in \cite[p. 398-400]{NS}). 
	On the other hand, on combining with the Chebotarev Density Theorem, one can obtain a lower bound for the product in Theorem \ref{thm:sieve} (2), and this results in an upper-bound for \eqref{eq:loweralmostprime} (with a smaller power on the log term). In \cite[\S2]{NS}, a stronger two-sided condition (e.g \cite[$\Omega_2(\kappa,L)$ p.~142]{HR}) is proven, under the extra assumption in Remark \ref{rmk:sieveremark} (ii), so that \eqref{eq:loweralmostprime} has expected magnitude of growth. As the weaker one-sided condition (2) is enough to execute the sieve and is sufficient to deduce Corollary \ref{cor:primepoint}, we do not seek to obtain best possible $r_0$ or to minimize the exponent $\beta_2$ on $\log T$ in our article.
\end{rem}

\section{Anisotropic case}\label{se:anisotropic}

The following is our main result of this section, whose proof shows how the fibration method and the affine linear sieve established in preceding sections match together.

\begin{thm}\label{th:mani1}
	Let $G$ be a three-dimensional semisimple simply connected linear algebraic group over a number field $k$ and $T\sbt G$ be a maximal torus.
	If $T(k_{\infty})$ is not compact, then $G$ satisfies \APSA off $\infty_{k}$.
\end{thm}

\subsection{Overview of the proof}\label{se:overview41}
	The isotropic case is already covered by Theorem \ref{thm:isotropic}, so we only need to consider the case where $G$ is \emph{anisotropic} in what follows.
Let $D\sbt G$ be closed of codimension at least $2$ and let $U:=G\setminus D$. 
The \textbf{(ASPA)} off $\infty_{k}$ for $G$ amounts to saying $U$ satisfies \SA off $\infty_{k}$ by Definition \ref{def:sa}.
In \textbf{Step I}, by using the results in Section \ref{se:isotropic}, we can choose a torus $T$ such that $D$ intersects transversally with any fibre of the quotient $\pi:G\to G/T$. Here we make essential use of the assumption $\dim G=3$. 
The utilization of $\pi$ allows us to construct ``good'' integral models (properties (i)--(v) in \textbf{Step II}) and control uniformly the fibres intersecting $D$, most of them being zero-dimensional.
In \textbf{Step III}, we first use the \SA property off $\infty_{k}$ for $G$ to produce rational points ($P$ in \textbf{Step IV}) in any open subset ($W_1$ \eqref{eq:W1} in the proof) of $G(\RA_{k}^{\infty_{k}})$. The goal is to find points avoiding $D$ modulo almost all places to ensure that they are in $U(\RA_{k}^\infty)$ (the open set $W$ \eqref{eq:W} in the proof). In \textbf{Step IV}, we apply the sieve results in Section \ref{se:sieve} to produce a new point $P^\prime$ in $W_1$,
so that the regular function $f$ (defined by \eqref{eq:f}) vanishing on $\pi^{-1}(\pi(D))$, when evaluated at $P^\prime$, is almost-prime and has only sufficiently big prime divisors (lying in $S_0$ \eqref{eq:S0} in the proof).
As the torus $T$ contains integral points of infinite order thanks to its non-compactness at the archimedean places, and the function $f$ is chosen to be compatible with the fibration $\pi$, in \textbf{Step V}, 
we make use of the action of $T$ on the fibre containing $P^\prime$ without changing the value of $f$ to find a new point $P^{\dprime}$ avoiding $D$ modulo all places in $S_0$. This point automatically avoids $D$ modulo any other places not in $S_0$, thereby achieving our goal.


\subsection{A lemma on actions of topological groups}
We include the following elementary result for the sake of completeness.
\begin{lem}\label{le:topogrp}
Let $H$ be a topological group which acts continuously on a topological space $Y$.
Then, for any compact open subset $V\subset Y$, there exists an open subgroup $F$ of $H$ such that $F\cdot V=V$.
\end{lem}
\begin{proof}
 We may assume that $V\neq\varnothing$, otherwise we take $F$ to be the identity. By the continuity assumption, for any $y\in V$, there exists an open subset $V_y\subset V$ of $Y$ and an open subset $H_y\subset H$ containing the identity such that $H_y\cdot V_y\subset V_y$. Since $V$ is compact, we can choose points $y_1,\cdots,y_n\in V$ such that $V=\cup_{i=1}^n V_{y_i}$. Then the open subgroup $F$ generated by $\cap_{i=1}^n H_{y_i}$ satisfies that $F\cdot V=V$.
\end{proof}

\subsection{Proof of Theorem \ref{th:mani1}}
Recall that our goal is to show that $U:=G\setminus D$ satisfies \SA off $\infty_{k}$, whenever $\codim(D,G)$ is at least $2$, under the assumption that $G$ is $k$-anisotropic. 

	\textbf{Step I: Setting up the fibration.}
	
 The assumption $\dim(G)=3$ implies that the group $G$ is $k$-simple (by considering the decomposition of $G$ into $k$-simple factors and by using the fact that there is no one-dimensional semisimple group). Moreover, the Bruhat decomposition implies that the big cell of $G_{\bar{k}}$ is of the form $U_+\times T_{\bar{k}}\times U_-$ where $U_+$ (resp. $U_-$) is the unipotent radical of some Borel subgroup (resp. opposite Borel subgroup) of $G_{\bar{k}}$ (cf. \cite[Cor. 14.14]{Bo}). So $\codim(T,G)$ is even and $\dim(T)=1$. 
 

Let us consider the quotient map $\pi: G\to Y:=G/T$. By Lemma \ref{isolem2}, upon replacing $T$ by another $k$-conjugate if necessary, we may assume there exists a closed subset $D'\sbt D$ such that
$T\cdot D'=D'$, $\dim(D')<\dim(D)$ and $\pi(G\setminus D)=Y\setminus \pi(D')$.
Since $\dim(G)=3$, one has $\dim(D')<\dim (D)\leqslant 1$ and so $D'=\varnothing$ because $T\cdot D'=D'$. 
As $\pi$ is a $T$-torsor, its restriction $\pi_U:=\pi|_U: U\to Y$ is smooth surjective with geometrically integral fibres. Thus $D$ does not contain any fibre of $\pi$ and, for any $y\in Y$, the fibre $D_y:=\pi^{-1}(y)\cap D$ is either empty or of $\dim(D_y)=0$.

Since $Y$ is affine (by Matsushima's criterion since $T$, $G$ are reductive, cf. \cite[\S4~Thm.~4.17]{PV}), and $\dim(\overline{\pi(D)})\leqslant \dim(D)<\dim Y=2$, we can find a regular function $F\in k[Y]$ so that its zero locus $E$ in $Y$ contains $\overline{\pi(D)}$. Let $V_0$ be the non-empty open subset $Y\setminus E$, and we have $\pi^{-1}(V_0)\sbt U$.

\textbf{Step II: Choosing integral models and locating ``bad'' fibres.}

We now fix integral models of $G,T$ and $U$ as follows.
We choose an embedding $G\sbt \GL_{n,k}$, and 
let $\CG$ be the integral closure of $G$ in $\GL_{n,\CO_k}$, $\CT$ (resp.~$\CD$) be the integral closure of $T$ (resp.~$D$) in $\CG$, and $\CU:=\CG\setminus \CD$.
Then $\CG$, $\CT$ are flat over $\CO_k$, and the quotient $$\pi: \CG\to \CY:=\CG/\CT$$ exists as a scheme over $\CO_k$ by a theorem of Raynaud (cf. \cite[\S4]{Ana}). Let $\CE$ be the integral closure of $E$ in $\CY$ and $\CV_0:=\CY\setminus \CE$.

To prove that $U$ satisfies $\SA$ off $\infty_{k}$, we need to show that, for $W\sbt U(\RA_k^{\infty_{k}})$ any non-empty open subset,
\begin{equation}\label{eq:WU}
	W\cap U(k)\neq\varnothing.
\end{equation}
First of all, there exists a finite set of places $S\sbt (\Omega_k\setminus \infty_{k})$ such that:

(i) by shrinking $W$ if necessary, we may assume 
\begin{equation}\label{eq:W}
W=\prod_{v\in S}W_v\times \prod_{v\notin S\cup \infty_{k}}\CU(\CO_v),
\end{equation}
with $W_v\sbt U(k_v)$ a non-empty compact open subset for each $v\in S$;

(ii) $\CO_{k,S}$ is a principal ideal domain and there exists a finite subset $S'\sbt \Omega_{\BQ}$ such that $\CO_{k,S}$ is finite étale over $\BZ_{S'}$;

(iii) the regular function $F$ defined at the end of \textbf{Step I} satisfies $F\in \CO_{k,S}[\CY]$; 

(iv)  for all $v\notin S\cup \infty_{k}$, the map $\pi|_{\CU}: \CU(\CO_v)\to \CY(\CO_v)$ is surjective, and $\CV_0(\CO_v)\neq\varnothing$ (by \cite[(ii) Proof of Thm. 4.5]{Conrad});

(v) there exists an integer $L$ such that, for any $y\in \CY\times_{\CO_k}\CO_{k,S}$, the fibre $\CD_y:=\CD\cap\pi^{-1}(y)$ is either empty, or has dimension $0$ (cf. \cite[Ex. II. 3.22]{Har}) and satisfies 
\begin{equation}\label{eq:N}
\# \CD_y(\overline{k(y)})\leqslant \deg (\CD_y)\leqslant L,
\end{equation}
thanks to the upper semi-continuity of the function (cf. \cite[Ex. II. 5.8]{Har})
$$y\mapsto \deg(\CD_y):=\dim_{k(y)}((\pi|_{\CD})_*(\CO_{\CD})\otimes_{\CO_{\CY}}k(y)).$$

We summarise the construction above to the following commutative diagram:
\[\xymatrix{\pi^{-1}(\CV_0)\ar@{^{(}->}[r]\ar[d] &\CU\ar@{^{(}->}[r]\ar[rd]_{\pi_U} & \CG\ar[d]^{\pi} & \CD\ar[d]\ar@{_{(}->}[l]\\
	\CV_0\ar@{^{(}->}[rr]&&\CY& \CE.\ar@{_{(}->}[l]
}\]
\textbf{Step III: First torus action.}

Recall the set $W$ \eqref{eq:W}. Let us consider
\begin{equation}\label{eq:W1}
W_1:= \prod_{v\in S}W_v\times \prod_{v\notin S\cup \infty_{k}}\CG(\CO_v)\subset G(\RA_{k}^{\infty_{k}}).
\end{equation}
By Lemma \ref{le:topogrp}, there exists, for each $v\in S$, an open compact subgroup $\Phi_v\sbt G(k_v)$ such that $\Phi_v\cdot W_v=W_v$. 
Consider subgroups
$$\Phi_G:=\prod_{v\in S}\Phi_v\times \prod_{v\notin S\cup \infty_{k}}\CG(\CO_v)\subset G(\RA_{k}^{\infty_{k}}),$$  and   $$\Phi_T:=\prod_{v\in S}(\Phi_v\cap\CT(\CO_v))\times \prod_{v\notin S\cup \infty_{k}} \CT(\CO_v)\subset \prod_{v\notin \infty_{k}}\CT(\CO_v).$$
We then have 
\begin{equation}\label{eq:W1fix}
\Phi_G\cdot W_1=W_1 \quad\text{and}\quad \Phi_T\cdot W_1=W_1.
\end{equation} 
By assumption, $G$ is anisotropic and $T(k_{\infty})$ is not compact, so the group $\CT(\CO_k)$ is infinite (cf. \cite[\S4.5~Cor.~1]{PR}).
Since $\Phi_T$ has finite index in $\prod_{v\notin \infty_{k}}\CT(\CO_v)$, the group $\Phi_T(\CO_k):=\CT(\CO_k)\cap \Phi_T$ has finite index in $\CT(\CO_k)$ and hence is infinite. Let $1_{\CT}$ be the identity and fix \begin{equation}\label{eq:Q}
Q\in \Phi_T(\CO_k)
\end{equation} an element of infinite order. For any $v\in \Omega_k\setminus \infty_{k}$, we denote by $\ord(Q\mod v)$ the order of the element $Q \mod v$ in the group $\CT(k(v))$. We want to show:

(vi) for any integer $r\in\BN_{\geqslant 1}$, there exists $M_r>0$ such that, for any finite place $v$ with $\#k(v)>M_r$, one has that $\ord(Q\mod v)\geqslant rL+1$ (recall $L$ \eqref{eq:N}).\footnote{The reason of choosing these numbers will be clear in \textbf{Step V}.}

Indeed, for any integer $l\in \BN_{\geqslant 1}$, consider the set
$$B(Q,l):=\{v\in \Omega_k\setminus \infty_{k}:\ord(Q\mod v)\leqslant l\}.$$ Then it is clear that $B(Q,l)=\cup_{1\leqslant j\leqslant l}B(Q^j)$, where each set $$B(Q^j):=\{v\in \Omega_k\setminus \infty_{k}:\ Q^j\equiv 1_{\CT}\mod v\}$$ has finite cardinality. Now to see the statement (vi) it suffices to take $l=rL$ and $M_r:=\max_{v\in B(Q^j),1\leqslant j\leqslant l}(\#k(v))$, so that for any $v\in\Omega_{k}\setminus\infty_{k}$ such that $\#k(v)>M_r$, we have $v\in\Omega_{k}\setminus(\infty_{k}\cup B(Q,rL))$, in other words, $\ord(Q\mod v)\geqslant rL+1$. This proves (vi).

\textbf{Step IV: Applying the affine sieve.}

Since $G(k_\infty)$ is not compact (because of $T(k_\infty)$), $G$ satisfies \SA off $\infty_{k}$ by Theorem \ref{rmk:KneserPlatonov}. Hence there exists 
$P\in G(k)\cap W_1$ (recall $W_1$ \eqref{eq:W1}). 
Let us define the regular function
\begin{equation}\label{eq:f}
f(g):=(F\circ \pi) (g\cdot P).
\end{equation}
Then $f\in \CO_{k,S}[\CG]$ by (iii). We claim that 
$f(\CG(\CO_v))\cap \CO_v^{\times}\neq \varnothing$ for any $v\notin S\cup \infty_{k}$. Indeed, according to (iv), we have $\pi(\CU(\CO_v))\cap \CV_0(\CO_v)\neq\varnothing$. Take any $P_0 \in\CU(\CO_v)$ such that $\pi(P_0)\in \CV_0(\CO_v)$, then the element $g_1:=P_0\cdot P^{-1}\in\CG(\CO_v)$ verifies $f(g_1)=F(\pi(P_0))\not\equiv 0\mod v$ by the definition of $\CV_0$. This proves the claim.
Taking (ii) into account, the hypotheses of Corollary \ref{cor:primepoint} are satisfied for the pair $(f,\Phi_G)$. We thus obtain:

(vii) there exists an integer $r_0$ and an element $g_0\in \Phi\cap \CG(\CO_{k,S})$, such that $f(g_0)$ has at most $r_0$ prime factors in $\CO_{k,S}$ and all of their residue fields have cardinalities larger than $M_{r_0}$, where the constant $M_{r_0}$ is defined by (vi) in \textbf{Step III} by taking $r=r_0$.

\textbf{Step V: Second torus action -- avoidance of ``bad'' fibres.}

Recall \eqref{eq:Q}. Let us define $P^\prime:=g_0\cdot P$, 
$$\Theta:=\{Q^l\cdot P^\prime:0\leqslant l\leqslant r_0\cdot L\}$$ and
\begin{equation}\label{eq:S0}
S_0:=\{v\in \Omega_{k}\setminus (S\cup\infty_{k}):\  \pi(P')\mod v\in \CE\}.
\end{equation}
Then $P^\prime\in\Theta\subset G(k)\cap W_1$ thanks to \eqref{eq:W1fix}.
And (vii) implies that the set $S_0$ (if non-empty) contains at most $r_0$ places.  
	
	For any $v\in S_0$, let $y_v:=\pi(P') \mod v\in \CY(k(v))$ and $$\Theta_v:=\{x\in \Theta:\ x\mod v \in \CD_{y_v}\}.$$
	(Recall $\CD:=\CG\setminus\CU$ defined in the beginning of \textbf{Step II} and $\CD_y$ in (v).)
	According to (vi) and (vii), for any $v\in S_0$, we have $\#k(v)>M_{r_0}$ and hence $\ord(Q\mod v)\geqslant r_0L+1$. 
	So the reduction map $\Theta\xrightarrow{\mod v} \CG_{y_v}(k(v))$ is injective.
	Thus $\#\Theta_v\leqslant \#\CD_{y_v}(k(v))\leqslant\deg \CD_{y_v}$.
	On the other hand, on combining \eqref{eq:N} in (v), we get
	$$\# \Theta= r_0L+1\geqslant 1+\sum_{v\in S_0} \deg \CD_{y_v} \geqslant 1+\sum_{v\in S_0} \#\Theta_v> \sum_{v\in S_0} \#\Theta_v .$$
	This means that we can find a point $P^{\dprime}\in\Theta$ such that $P^{\dprime}\mod v\notin \CD$ for all $v\in S_0$.

By the definition of $S_0$, $\pi(P^\prime)=\pi(P^{\dprime})\mod v\not\in \CE$ for any $v\not\in S\cup S_0\cup \infty_{k}$, and therefore for any such $v$, we have $P^{\dprime}\mod v\notin \CD$.
We have finally proven that $P^{\dprime}\mod v\notin \CD$ for any $v\notin S\cup \infty_{k}$, and therefore $P^{\dprime}\in W\cap G(k)=W\cap U(k)$ (recall $W$ \eqref{eq:W}), so \eqref{eq:WU} is achieved. This finishes the proof of the theorem.
\qed

\section{Affine quadrics}\label{se:affquad}
In this section we complete the proof of Theorem \ref{thm:anisotropic} and then deduce Theorem \ref{cor:anisotropic}, finally confirming \APSA off $\infty_{k}$ for affine quadrics.
\subsection{Proof of Theorem \ref{thm:anisotropic}}
By Theorem \ref{thm:reduction}, it suffices to show that $G^\prime$ satisfies \APSA off $\infty_{k}$. Upon replacing $G$ by $G'$,  we may assume that $\dim(G)=3$.
Since $G(k_{\infty})$ is not compact, there exists $v\in \infty_{k}$ such that $G(k_v)$ is not compact.
By \cite[Thm. 3.1]{PR}, $G_{k_v}$ is isotropic, so it contains an isotropic maximal torus $T_v\sbt G_{k_v}$ over $k_v$.
By \cite[\S 7.1, Cor. 3]{PR}, there exists a maximal torus $T\sbt G$ over $k$ such that $T_{k_v}$ is conjugate to $T_v$ over $k_v$. 
Thus $T(k_{\infty})$ is not compact. 
Then the statement follows from Theorem \ref{th:mani1}.
\qed

\subsection{Proof of Theorem \ref{cor:anisotropic}}
To prove (i), we may assume $q(x_1,\cdots, x_n)=\sum_{i=1}^na_ix_i^2$ with $a_i\in k^{\times}$ for all $i$. Since $q$ is isotropic over some $v_0\in \infty_{k}$, if $v_0$ is real, then we may assume that $v_0(a_1\cdot a_2)< 0$.
Then the spin group 
$$G':=\operatorname{Spin}(a_1x_1^2+a_2x_2^2+a_3x_3^2)$$
is a three dimensional closed subgroup of $G=\operatorname{Spin}(q)$ and is not compact over $v_0$.
 If $G$ is $k$-simple, then the statement (i) now follows from Theorem \ref{thm:anisotropic}. The only case where $G$ is not $k$-simple (cf. \cite[Proof of Thm.~6.1]{CTX09}) is when $n=4$ and $\det(q)\in k^{\times 2}$, and in that case $G\simeq_k H\times H$, where $H$ is a $k$-form of $\SL_{2,k}$ and $H_K\simeq \SL_{2,K}$ over an extension $K$ of $k$ if $q$ is isotropic over $K$. Since we assume that $q$ is isotropic over the place $v_0$, 
so  $G_{k_{v_0}}\simeq_{k_{v_0}}\SL_{2,k_{v_0}}\times\SL_{2,k_{v_0}}$.
Therefore the three dimensional $k$-simple group $H$ is not compact over $v_0$. It then satisfies \APSA off $\infty_{k}$ by Theorem \ref{thm:anisotropic}. The fact that $G$ also verifies \APSA off $\infty_{k}$ in this case follows from \cite[Prop. 4.7]{CLX}.

To prove (ii), we may assume that the affine quadric $V:q(x_1,\cdots,x_n)=a_0$ satisfies $V(\RA_k)\neq\varnothing$. Then $V(k)\neq\varnothing$ by the Hasse–Minkowski theorem. Then the statement follows from \cite[Thm.~1.3]{CLX}, since by \cite[\S5.3]{CTX09}, when $n\geqslant 4$, $V$ is a homogeneous space under $\operatorname{Spin}(q)$ with stabilizer isomorphic to $\operatorname{Spin}(h)$ where $h$ is another non-degenerate quadratic form in $n-1$ variables, and we have 
$\overline{k}[V]^\times=\overline{k}^\times, \operatorname{Br}(V)=\operatorname{Br}(k)$.
\qed

\section*{Acknowledgements}
We would like to thank Jean-Louis Colliot-Thélène, Cyril Demarche, Philippe Gille, Yonatan Harpaz, Diego Izquierdo, Olivier Wittenberg and Fei Xu for many interesting discussions, and to Dasheng Wei and Han Wu for kindly pointing out some inaccuracies. We are grateful to Tim Browning for his interest in this project and to Ulrich Derenthal for his generous support. Part of this work was carried out at the Institut Henri Poincaré during the trimester ``À la redécouverte des points rationnels''. We are grateful to the organizers and the staffs for creating very stimulating atmosphere. We would like to express our heartfelt thanks to the anonymous referees for their careful scrutiny and valuable suggestions. The first author is supported by a Humboldt-Forschungsstipendiaten. The second author is supported by grant DE 1646/4-2 of the Deutsche Forschungsgemeinschaft.

\bibliographystyle{alpha}
\end{document}